\documentclass[a4paper,10pt]{article}
\usepackage[utf8]{inputenc}
\usepackage{amsmath,amssymb,amsthm}
\usepackage[colorlinks,citecolor=blue]{hyperref}
\usepackage{cleveref}
\usepackage{enumitem}
\usepackage{mathrsfs}
\usepackage{color, colortbl}
\usepackage{graphicx}
\usepackage{hhline}
\usepackage{tikz}
\usetikzlibrary{arrows.meta}
\usetikzlibrary{backgrounds}

\newcommand{\QQ}{\mathbb{Q}}

\newcommand{\ran}{\mathscr{R}^-}
\newcommand{\rap}{\mathscr{R}^+}

\newcommand*\pFq[4]{{}_{#1}F_{#2}\biggl(\genfrac..{0pt}{}{#3}{#4};1\biggr)}

\newtheorem{theorem}{Theorem}[section]

\newtheorem{lemma}[theorem]{Lemma}

\title{Ramanujan-Bernoulli numbers as moments of Racah polynomials}
\author{F. Chapoton}
\date{\today}

\begin{document}

\maketitle

\begin{abstract}
  The classical sequence of Bernoulli numbers is known to the the
  sequence of moments of a family of orthogonal polynomials. Some
  similar statements are obtained for another sequence of rational
  numbers, which is similar in many ways to the Bernoulli numbers.
\end{abstract}

\section*{Introduction}

Let us consider the following sequence of rational numbers
\begin{equation*}
  \label{nombres_plus}
  (\rap_n)_{n\geq 0} = 1, \frac{1}{3}, \frac{1}{30}, -\frac{1}{105}, \frac{1}{210}, -\frac{1}{231},
  \frac{191}{30030}, -\frac{29}{2145}, \frac{2833}{72930}, -\frac{140051}{969969}, \dots
\end{equation*}
and the almost identical companion sequence
\begin{equation*}
  \label{nombres_moins}
  (\ran_n)_{n\geq 0} = 1, -\frac{1}{6}, \frac{1}{30}, -\frac{1}{105}, \frac{1}{210}, -\frac{1}{231},
  \frac{191}{30030}, -\frac{29}{2145}, \frac{2833}{72930}, -\frac{140051}{969969}, \dots
\end{equation*}
that differs only by the second term. These sequences
are very close to the classical sequence of Bernoulli numbers, but not so well known.

The sequence $\ran$ seems to have first appeared, in a slightly
implicit way and up to an easy power of $2$, in an article of Ludwig Seidel
from 1877 \cite{seidel}, as the main diagonal of the difference table
of the Bernoulli numbers. This diagonal is highlighted using a bold
font in the table given there on page 181. Seidel proves that the two
diagonals below and above the main diagonal in the same table are
essentially given by the same sequence, up to multiplication by $2$.

Later, the first terms of the sequence $\ran$ appeared explicitly in
one of Srinivasa Ramanujan's notebooks (written between 1903 and
1914), as giving, up to a simple factor, the first coefficients in an
unusual asymptotic expansion for the harmonic numbers into powers of
the inverse of the triangular numbers $\binom{n+1}{2}$. This is
displayed in Bruce Berndt's edition of Ramanujan's notebooks as the
number (9) of \cite[Chapter 38]{berndtV}). It can also be noted that
the first term of this asymptotic expansion of Ramanujan was obtained
by Ernesto Cesàro in 1885 in \cite{cesaro}. The first complete proof
that the sequence $\ran$ describes this full asymptotic expansion was
given by Mark Villarino in \cite{villarino}, where an historical
account can be found.

The sequence $\ran$ has been considered again in 2005 by Kwang-Wu Chen
\cite{chenKW}, from a point of view close to that of Seidel. He
obtained a functional equation and a continued fraction for their
generating series. We will be more precise about his results later in
section \ref{main}.

The sequence $\rap$ has recently surfaced in a very different
algebraic context \cite{ch_series1, ch_series2} related to the notion
of pre-Lie algebra. There is a complete algebraic theory of
tree-indexed series, very similar to usual power series in one
variable, but where monomials are indexed by finite rooted
trees. These tree-indexed series can be multiplied (in a
non-associative way) but also composed (in an associative way). One
can therefore consider the group of tree-indexed series that are
invertible for the composition. This group contains a special element
$A$ with rational coefficients, which is a kind of tree-exponential
and has very simple coefficients. Its inverse $\Omega$ has more subtle
and interesting coefficients, among which the Bernoulli numbers $B_n$
for corollas and the numbers $\rap_n$ for another sequence of rooted
trees.

After these works on tree-indexed series, it has been understood in
\cite{ch_essouabri} that the sequence $\rap$ also appears in the values
at negative integers of some kind of non-standard $L$-function. More
details will be given in section \ref{L_function}.

The main aim of the present article, apart from advertising the
sequences $\rap$ and $\ran$, is to describe a new relationship between
these sequences and the moments of some classical families of
hypergeometric orthogonal polynomials, namely Racah polynomials. This
relationship in particular implies nice corollaries about continued
fractions and Hankel determinants, by the general theory of orthogonal
polynomials. We will not say more about this, because the exact
statements can be easily reconstructed.

This is very similar to the known relationship between the Bernoulli
numbers and another family of hypergeometric orthogonal polynomials,
namely Hahn polynomials. This is therefore still another way in which the
sequence $\rap$ is comparable with the Bernoulli numbers.

One word about terminology: there does not seem to be any accepted name for the sequences $\ran$ and $\rap$. The name ``median Bernoulli numbers'' is used for the diagonal of the difference table of Bernoulli numbers, which differs from $\ran$ by powers of $2$. We propose that the name of ``Ramanujan-Bernoulli numbers'' may be suitable for the sequence $\ran$ itself.

% The numerators of $\rap_n$ are given by \oeis{A238813}.

% -A181131: Denominator pour median bernoulli

% -A190339: The denominators for median bernoulli

% -A238813, A093334

% -A212196: Numerators of the Bernoulli median numbers

% -A181130 is an unsigned version with offset 1.

%EXTRAIT DE OEIS A093334 : a few starting numerical terms were given by
%Euler ? and Ramanujan; the form of the general term and the
%behavior of the series were determined by Villarino \cite{villarino}.
%
% pas fiable ; et ou chercher dans l'immense oeuvre d'Euler ?

Let us end this introduction by some open questions.

First, it seems that the sequence $\ran$ alternates in sign. To the
best of our knowledge, this is not yet proved. It is also expected that
the associated non-standard $L$-function considered in section
\ref{L_function} has a simple zero between consecutive negative
integers.

There may exist $q$-analogues for some of the results of this
article, in the spirit of the continued fractions for
the $q$-Bernoulli numbers of Carlitz studied in \cite{ch_zeng}, but
they have so far remained elusive.

One can also wonder, in a very wild speculation, if there is, for the
sequences $\rap$ or $\ran$, something like the relationship between
Bernoulli numbers and the algebraic K-theory of the ring of integers.

\section{Definitions}

\subsection{Sequences $\ran$ and $\rap$}

\label{defini}

Let us now give the formal definition of $\ran$ and $\rap$.

Let us first introduce the classical Bernoulli numbers $B_n$, defined by
\begin{equation}
  \label{eq:ber}
  \frac{t}{e^t-1} = \sum_{n\geq 0} B_n \frac{t^n}{n!}.
\end{equation}

Define a linear form $\Psi$ on the vector space
$\QQ[x]$ of polynomials in one variable $x$ with rational coefficients, by
\begin{equation}
  \Psi(x^n) = B_n,
\end{equation}
for all $n \geq 0$.

The sequence $\rap$ is defined by
\begin{equation}
  \rap_n = \Psi\left(\binom{x+2}{2}^n\right),
\end{equation}
and the companion sequence $\ran$ by
\begin{equation}
  \ran_n = \Psi\left(\binom{x+1}{2}^n\right),
\end{equation}
for all $n \geq 0$. For example,
\begin{equation*}
  \rap_6 = \frac{191}{30030} = 1 -\frac{9}{2} + \frac{49}{8} -\frac{157}{32} + \frac{8989}{2688} -\frac{157}{128} + \frac{245}{1408} -\frac{691}{174720}.
\end{equation*}
and
\begin{equation*}
  \ran_6 = \frac{191}{30030} = \frac{1}{2688} -\frac{1}{128} + \frac{25}{1408} -\frac{691}{174720}.
\end{equation*}

The fact that $\rap_n$ and $\ran_n$ are the same except when $n=1$
follows from the next lemma.
\begin{lemma}
  \label{ran_rap}
  For every $n \not=1$, $\Psi(x^n(x+1)^n) = \Psi((x+1)^n(x+2)^n)$.
\end{lemma}
\begin{proof}
  The case $n=0$ is trivial, so that one can assume $n \geq 2$. Let us
  compute the difference
  \begin{equation*}
    ((x+1)(x+2))^n - (x(x+1))^n = 2 (x+1)^n \left( \sum_{\substack{k=0\\k\equiv 1 (2)}}^n \binom{n}{k} (x+1)^{n-k}          \right).
  \end{equation*}
  This is a linear combination
  of odd powers of $(x+1)$, excluding $x+1$. The image by $\Psi$
  therefore vanishes, because the involved Bernoulli numbers are zero.
\end{proof}

\subsection{Racah's orthogonal polynomials}

The orthogonal polynomials of Racah are defined by (see \cite[\S 1.2]{koekoek_report}):
\begin{equation}
  \label{racah_hyper}
  R_n(\lambda(x)\,;\,\alpha, \beta, \gamma, \delta)
  = \pFq{4}{3}{-n, n+\alpha+\beta+1,-x,x+\gamma+\delta+1}{\alpha+1,\beta+\delta+1,\gamma+1},
\end{equation}
where $\lambda(x) = x (x+\gamma+\delta+1)$ and using the standard
notation for hypergeometric functions. The parameters
$\alpha,\beta,\gamma$ and $\delta$ will remain implicit in all the
notations.

Using the usual Pochhammer symbol $(x)_k = x(x+1)\cdots(x+k-1)$, the
hypergeometric series above is given explicitly by the finite sum
\begin{equation}
  \label{racah_explicit}
  \sum_{k = 0}^{n} \frac{(-n)_k (n+\alpha+\beta+1)_k (-x)_k (x+\gamma+\delta+1)_k}{(\alpha+1)_k (\beta+\delta+1)_k (\gamma+1)_k k!}.
\end{equation}

The polynomials $R_n$ are not monic in general. Setting
\begin{equation}
  R_n(x) = \frac{(n+\alpha+\beta+1)_n}{(\alpha+1)_n (\beta+\delta+1)_n (\gamma+1)_n} p_n(x),
\end{equation}
one obtains the corresponding family of monic orthogonal polynomials.

By the general theory of orthogonal polynomials and in particular by
Favard's lemma, there exists a unique linear form $\Lambda$ on the
vector space of polynomials in $x$ such that $\Lambda(1) = 1$ and
$\Lambda(p_n) = 0$ for $n > 0$. Then the moments of the family of
orthogonal polynomials are given by $\Lambda(x^n)$ for $n\geq 0$.

Note that the linear form $\Lambda$ is also characterized by
$\Lambda(1)=1$ and $\Lambda(R_n)=0$ for $n > 0$. Hence there is no
need to consider the monic orthogonal polynomials.

\section{Main theorems}

\label{main}

In this section, we will state five similar theorems, saying that some
sequences are the sequences of moments of some specific families of
Racah polynomials.

\begin{theorem}
  \label{th_nega_zero}
  The numbers $(2^n \ran_n)_{n \geq 0}$ are the moments of the orthogonal
  polynomials $R_n$ of parameters
  $(\alpha, \beta, \gamma, \delta) = (0, -1/2, 0, 0)$.
\end{theorem}
\begin{proof}
  We want to prove that
  \begin{equation*}
    \mu_n = \Lambda(x^n) = \Psi(x^n(x+1)^n) = 2^n \ran_n,
  \end{equation*}
  for all $n \geq 0$. By the characterisation of the linear form
  $\Lambda$, one just needs to check that $\Psi(R_n(x(x+1))$ does
  vanish when $n > 0$ and takes the value $1$ at $n=0$. The second
  point is clear because $R_0=1$. We can use the explicit hypergeometric expression
  \eqref{racah_explicit} at the given parameters:
  \begin{equation*}
    R_n(x(1+x)) = \sum_{k=0}^{n} \frac{(-n)_k (n+1/2)_k (-x)_k (x+1)_k}{(1/2)_k k!^3}.
  \end{equation*}
  Applying $\Psi$ and using \Cref{eval_psi} with parameters
  $(d,e,i,j)=(k,k,k-1,k)$ gives
  \begin{equation*}
    \sum_{k=0}^{n} \frac{(-n)_k (n+1/2)_k}{(1/2)_k k! (2k+1)} = \pFq{2}{1}{-n,n+1/2}{3/2},
  \end{equation*}
  which is indeed $0$ for every $n \geq 1$ by the Chu-Vandermonde identity.
\end{proof}

We could get rid, in this theorem and the next four ones, of the
factor $2^n$ in the sequences of moments, by scaling the variable in
the orthogonal polynomials. Using the classical relationship between
orthogonal polynomials, continued fractions and Hankel determinants
(see for example \cite{ch_zeng}), one can deduce from the previous
theorem (and similarly for the next four theorems) a nice continued
fraction for the ordinary generating series of the sequence $\ran$,
and an explicit factorisation of the Hankel determinants made from
$\ran$. Both these results have been proved before by Kwang-Wu Chen in
\cite[\S 5]{chenKW}, without the connection with orthogonal
polynomials.

\begin{theorem}
  \label{th_nega_un}
  The numbers $(2^{n} \ran_{n+1}/\ran_1)_{n \geq 0}$ are the moments of the orthogonal polynomials $R_n$ of parameters $(\alpha, \beta, \gamma, \delta) = (-1/2,1,0,0)$.
\end{theorem}
\begin{proof}
  We want to prove that
  \begin{equation*}
    \mu_n = \Lambda(x^n) = \frac{1}{2 \ran_1}\Psi(x^{n+1}(x+1)^{n+1}) = 2^{n} \ran_{n+1} / \ran_1,
  \end{equation*}
  for all $n \geq 0$. By the characterisation of the linear form $\Lambda$, one
  just needs to compute
  \begin{equation*}
    \Psi(x(x+1)R_n(x(x+1))),
  \end{equation*}
  whose value at $n=0$ is $2 \ran_1$. We can use
  the explicit hypergeometric expression \eqref{racah_explicit} at the given
  parameters:
  \begin{equation*}
    R_n(x(x+1)) = \sum_{k=0}^{n} \frac{(-n)_k (n+3/2)_k (-x)_k (x+1)_k}{(1/2)_k (k+1)! k!^2},
  \end{equation*}
  and therefore
  \begin{equation*}
    x (x+1) R_n(x(x+1)) = - \sum_{k=0}^{n} \frac{(-n)_k (n+3/2)_k (-x-1)_{k+1} (x)_{k+1}}{(1/2)_k (k+1)! k!^2}.
  \end{equation*}
  Applying $\Psi$ and using \Cref{eval_psi} with parameters
  $(d,e,i,j)=(k+1,k+1,k-1,k)$ gives
  \begin{equation*}
    - \sum_{k=0}^{n} \frac{(-n)_k (n+3/2)_k}{(1/2)_k k! (2k+3)(2k+1)} = 
 (-1/3)\, \pFq{2}{1}{-n,n+3/2}{5/2},
  \end{equation*}
  which is once again $0$ for every $n \geq 1$ by the Chu-Vandermonde identity.
\end{proof}

\begin{theorem}
  \label{th_posi_un}
  The numbers $(2^n \rap_{n+1}/\rap_1)_{n \geq 0}$ are the moments of the orthogonal polynomial $R_n$ of parameters $(\alpha, \beta, \gamma, \delta) = (0,1/2,0,-2)$.
\end{theorem}
\begin{proof}
  We want to prove that
  \begin{equation*}
    \mu_n = \Lambda(x^n) = \frac{1}{2 \rap_1}\Psi((x+1)^{n+1}(x+2)^{n+1}) = 2^{n} \rap_{n+1} / \rap_1,
  \end{equation*}
  for all $n \geq 0$. By the characterisation of the linear form $\Lambda$, one
  just needs to compute
  \begin{equation*}
    \Psi((x+1)(x+2)R_n((x+1)(x+2))),
  \end{equation*}
  whose value at $n=0$ is $2 \rap_1$. We can use
  the explicit hypergeometric expression \eqref{racah_explicit} at the given
  parameters:
  \begin{equation*}
    R_n(x(x-1)) = \sum_{k=0}^{n} \frac{(-n)_k (n+3/2)_k (-x)_k (x-1)_k}{(-1/2)_k k!^3},
  \end{equation*}
  and therefore
  \begin{multline*}
    (x+1) (x+2) R_n((x+1)(x+2)) =\\ \sum_{k=0}^{n} \frac{(-n)_k (n+3/2)_k (x+1)(x+2) (-x-2)_{k} (x+1)_{k}}{(-1/2)_k k!^3}.
  \end{multline*}
  Applying $\Psi$ and using \Cref{eval_psi_deux} at $k$ gives
  \begin{equation*}
    -2 \sum_{k=0}^{n} \frac{(-n)_k (n+3/2)_k}{(-1/2)_k k! (2k+3)(2k+1)(2k-1)} =
   (2/3)\, \pFq{2}{1}{-n,n+3/2}{5/2},
  \end{equation*}
  which is once again $0$ for every $n \geq 1$ by the Chu-Vandermonde identity.
\end{proof}

Recall that $\rap_n = \ran_n$ for every $n \geq 2$ by \Cref{ran_rap}. The two following statements therefore also hold with $\ran$ replacing $\rap$.

\begin{theorem}
  \label{th_posi_deux}
  The numbers $(2^n \rap_{n+2}/\rap_{2})_{n \geq 0}$ are the moments of the orthogonal polynomials $R_n$ of parameters $(\alpha, \beta, \gamma, \delta) = (-1/2,2,1,-1)$.
\end{theorem}
\begin{proof}
  We want to prove that
  \begin{equation*}
    \mu_n = \Lambda(x^n) = \frac{1}{4 \rap_2}\Psi((x+1)^{n+2}(x+2)^{n+2}) = 2^{n} \rap_{n+2} / \rap_2,
  \end{equation*}
  for all $n \geq 0$. By the characterisation of the linear form $\Lambda$, one
  just needs to compute
  \begin{equation*}
    \Psi((x+1)^2(x+2)^2R_n((x+1)(x+2))),
  \end{equation*}
  whose value at $n=0$ is $4 \rap_2$. We can use the explicit
  hypergeometric expression \eqref{racah_explicit} at the given
  parameters:
  \begin{equation*}
    R_n((x+1)(x+2)) = \sum_{k=0}^{n} \frac{(-n)_k (n+5/2)_k (-x-1)_k (x+2)_k}{(1/2)_k (k+1)!^2 k!},
  \end{equation*}
  and therefore
  \begin{multline*}
    (x+1)^2 (x+2)^2 R_n((x+1)(x+2)) =\\ - \sum_{k=0}^{n} \frac{(-n)_k (n+5/2)_k (x+1)(x+2) (-x-2)_{k+1} (x+1)_{k+1}}{(1/2)_k (k+1)!^2 k!}.
  \end{multline*}
  Applying $\Psi$ to this expression and using \Cref{eval_psi_deux} at $k+1$ gives
  \begin{equation*}
    2 \sum_{k=0}^{n} \frac{(-n)_k (n+5/2)_k}{(1/2)_k k! (2k+1)(2k+3)(2k+5)} =
  (2/15)\, \pFq{2}{1}{-n,n+5/2}{7/2},
  \end{equation*}
  which is once again $0$ for every $n \geq 1$ by the Chu-Vandermonde identity.
\end{proof}

\begin{theorem}
  The numbers $(2^n \rap_{n+3}/\rap_3)_{n \geq 0}$ are the moments of the orthogonal
  polynomials $R_n$ of parameters
  $(\alpha, \beta, \gamma, \delta) = (2,1/2,2,0)$ evaluated at $x-2$.
\end{theorem}
\begin{proof}
  We want to prove that
  \begin{equation*}
    \mu_n = \Lambda(x^n) = \frac{1}{8 \rap_3}\Psi((x+1)^{n+3}(x+2)^{n+3}) = 2^{n} \rap_{n+3} / \rap_3,
  \end{equation*}
  for all $n \geq 0$. The linear form $\Lambda$ is characterized by the conditions
  \begin{equation*}
    \Lambda(R_n(x-2)) = 0
  \end{equation*}
  for all $n > 0$ and $1$ for $n=0$. We must therefore compute
  \begin{equation*}
    \Psi((x+1)^3(x+2)^3 R_n((x+1)(x+2)-2)),
  \end{equation*}
  whose value at $n=0$ is $8 \rap_3$. 

  Because $x(x+3) = (x+1)(x+2)-2$, one can use
  the explicit hypergeometric expression \eqref{racah_explicit} at the given
  parameters:
  \begin{equation*}
    R_n(x(x+3)) = \sum_{k=0}^{n} \frac{(-n)_k (n+7/2)_k (-x)_k (x+3)_k}{(3/2)_k (3)_k (3)_k k!},
  \end{equation*}
  and therefore
  \begin{multline*}
    (x+1)^3 (x+2)^3 R_n(x(x+3)) =\\ 4 \sum_{k=0}^{n} \frac{(-n)_k (n+7/2)_k (x+1) (x+2) (-x-2)_{k+2} (x+1)_{k+2}}{(3/2)_k (k+2)!^2 k!}.
  \end{multline*}
   Applying $\Psi$ to this expression and using \Cref{eval_psi_deux} at $k+2$  gives
  \begin{equation*}
    -8 \sum_{k=0}^{n} \frac{(-n)_k (n+7/2)_k}{(3/2)_k k! (2k+3)(2k+5)(2k+7)} =
  (-8/105)\, \pFq{2}{1}{-n,n+7/2}{9/2},
  \end{equation*}
  which is once again $0$ for every $n \geq 1$ by the Chu-Vandermonde identity.
\end{proof}

Up to the same factor of $2^n$, the sequence of polynomials in the
variable $u$ defined by
\begin{equation}
  \Psi(((x+u)(x+1-u))^n) = \Psi((x(x+1)+u(1-u))^n)
\end{equation}
can also be realised as a sequence of moments, by shifting by $u(1-u)$
the variable in the orthogonal polynomials used in \Cref{th_nega_zero}.

\section{Values of a non-standard $L$-function}

\label{L_function}

Let us consider the following analytic function
\begin{equation}
  \label{L2}
  L_2(s) = \sum_{n \geq 0} \frac{n+3/2}{\binom{n+2}{2}^s}.
\end{equation}

This has been studied in \cite{ch_essouabri}, as the special case
$P = \binom{x+2}{2}$ of the series
\begin{equation}
  \label{LP}
  L_P(s) = \sum_{n\geq 0} \frac{P'(n)}{P(n)^s}
\end{equation}
attached to polynomials with no roots at positive integers.

The formula \eqref{LP} is convergent in the right half-plane $\Re(s) > 1$. It is
known that $L_P$ admits an analytic continuation to a meromorphic
function on the entire complex plane, with only a simple pole at $1$
with residue $1$. Moreover, the values of $L_P$ at negative integers are rational numbers, given by
\begin{equation*}
  L_P(1-n) = -\Psi\left(P^n\right) / n,
\end{equation*}
for all integers $n \geq 1$. Here $\Psi$ is the linear form defined in \S \ref{defini}.

Therefore, the numbers $\rap_n$ are closely tied with values of $L_2$ at negative integers:
\begin{equation}
  L_2(1-n) =  - \rap_n / n,
\end{equation}
for all integers $n \geq 1$.

% The values of $L_2$ at positive integers are integral linear combinations of values of $\zeta$ at odd integers. 

Let us conclude this section by a short remark on the analytic
continuation of the functions $L_P$, already proved in
\cite{ch_essouabri}. Here we give another sketch of argument for the analytic
continuation to a barely larger right half-plane, similar to a classical
argument for the zeta function, and useful for numerical computations.

Let us define a polynomial
$A(n)$ by the properties that $A(0)=1$ and
\begin{equation}
  A(n+1) - A(n) = P'(n)
\end{equation}
for all $n \geq 0$. Then
\begin{equation}
  L_P(s) - \frac{1}{s-1} = \sum_{n=0}^{\infty} \frac{P'(n)}{P(n)^s} - \int_{1}^{\infty} x^{-s} dx
\end{equation}
can be rewritten as
\begin{equation}
  \sum_{n=0}^{\infty} \int_{A(n)}^{A(n+1)} \frac{1}{P(n)^s} dx
- \sum_{n=0}^{\infty} \int_{A(n)}^{A(n+1)} x^{-s} dx.
\end{equation}
Collecting the terms, one gets
\begin{equation}
  \sum_{n=0}^{\infty} \int_{A(n)}^{A(n+1)} \left(P(n)^{-s} - x^{-s}\right) dx.
\end{equation}
Now the term of index $n$ can be bounded in such a way as to imply
that the sum is convergent when $\Re(s) > 1-1/d$, where $d$ is the
degree of $P$.

% GIVES a way to compute the analogue of $\gamma$ ?
% voir prolongement_zeta.py

\appendix
\section{Evaluation lemmas}

This appendix contains two useful lemmas on the values of the
linear form $\Psi$ on specific families of polynomials.

\begin{lemma}
  \label{eval_psi}
  For all integers $d \geq 0$, $e \geq 0$, $0 \leq i \leq d-1$ and
  $0 \leq j \leq e$,
  \begin{equation}
    \Psi\left(\binom{-x+i}{d}\binom{x+j}{e}\right) = (-1)^{d+e-i-j-1} \frac{1}{(d+e+1) \binom{d+e}{i+j+1}}.
  \end{equation}
\end{lemma}
\begin{proof}
  This follows directly from the know fact (see \cite[Lemme
  1.2]{ch_essouabri}) that, for all integers $0 \leq i \leq d$ and $0 \leq j \leq e$,
  \begin{equation}
    \label{CE_lemme12}
    \Psi\left(\binom{x+i}{d}\binom{x+j}{e}\right) = (-1)^{d+e-i-j} \frac{1}{(d+e+1) \binom{d+e}{d-i+j}}.
  \end{equation}  
\end{proof}

\begin{lemma}
  \label{eval_psi_deux}
  For $k \geq 0$, there holds
  \begin{equation}
    \Psi\left(\binom{x+2}{2}\binom{-x-3+k}{k}\binom{x+k}{k}\right) = -\frac{1}{(2k+3)(2k+1)(2k-1)}
  \end{equation}
\end{lemma}

\begin{proof}
  Note that it is equivalent to compute
  \begin{equation*}
    \label{alternatif}
    (-1)^k \Psi\left(\binom{x+2}{2}\binom{x+2}{k}\binom{x+k}{k}\right).
  \end{equation*}
  Let us start by expanding the first product of binomials as
  \begin{equation*}
    \binom{x+2}{2}\binom{x+2}{k} = \sum_{2 \leq \ell\leq 4}\binom{2}{\ell-2}\binom{k}{\ell-2}\binom{x+\ell}{2+k},
  \end{equation*}
  by a general formula (see \cite[Proposition 2]{ch_essouabri}). Then \eqref{alternatif} becomes
  \begin{equation*}
    (-1)^k \sum_{2 \leq \ell\leq 4} \binom{2}{\ell-2}\binom{k}{\ell-2} \Psi\left(\binom{x+\ell}{2+k} \binom{x+k}{k}\right)
  \end{equation*}
  which can be evaluated using \eqref{CE_lemme12} as
  \begin{equation*}
    \sum_{2 \leq \ell\leq 4} (-1)^{\ell} \frac{ \binom{2}{\ell-2}\binom{k}{\ell-2} } {(3 + 2 k)\binom{2+2k}{\ell}}.
  \end{equation*}
  This can be expanded and simplified into the expected result.
\end{proof}

% NOTE: the simplest q-analogue of \eqref{alternatif} is not true.

\bibliographystyle{alpha}
\bibliography{racah.bib}

\end{document}